\documentclass[oneside,11pt]{amsart}

\usepackage{amssymb,amsfonts,amsmath,amsthm}
\usepackage{mathtools}
\usepackage{needspace}
\usepackage{lmodern}
\usepackage{microtype}
\usepackage[left=3cm,right=3cm,bottom=3cm]{geometry}
\usepackage{enumerate}
\usepackage{url}
\usepackage[colorlinks]{hyperref} 
\hypersetup{
bookmarksnumbered,
pdfstartview={FitH},
breaklinks=true,
linkcolor=blue,
urlcolor=blue,
citecolor=blue,
bookmarksdepth=2
}
\usepackage[nameinlink,capitalise,noabbrev]{cleveref}
\allowdisplaybreaks

\newtheorem{thm}{Theorem}[section]

\newtheorem*{thm*}{Theorem}
\newtheorem*{cor*}{Corollary}
\newtheorem*{prop*}{Proposition}

\newtheorem{prop}[thm]{Proposition}
\newtheorem{lem}[thm]{Lemma}

\theoremstyle{definition}

\newtheorem*{notn*}{Notation}

\theoremstyle{remark}
\newtheorem{rem}[thm]{Remark}

\newtheorem*{idea*}{Idea}

\newcommand\smvee{\raise0.9ex\hbox{$\scriptscriptstyle\vee$}}
\makeatletter
\let\origsection\section
\renewcommand\section{\@ifstar{\starsection}{\nostarsection}}

\newcommand\nostarsection[1]
{\sectionprelude\origsection{#1}\sectionpostlude}

\newcommand\starsection[1]
{\sectionprelude\origsection*{#1}\sectionpostlude}

\newcommand\sectionprelude{%
  \vspace{1em}
}

\newcommand\sectionpostlude{%
  \vspace{1em}
}

\newcommand*\mybackmatter{%
\startcontents
\phantomsection
\addcontentsline{toc}{part, thebibliography}{}%
\@endpart}

\makeatother

\makeatletter
\let\c@equation\c@thm
\makeatother
\numberwithin{thm}{section}
\numberwithin{equation}{section}

\title[Non-algebraicity of Deligne--Hitchin twistor spaces]{Non-algebraicity of Deligne--Hitchin twistor spaces}
\author{Pengfei Huang}

\begin{document}
\pagenumbering{arabic}

\maketitle

\begin{abstract}
In this paper, we prove that on a smooth complex projective curve of positive genus, the Deligne--Hitchin twistor spaces for the structure groups $\mathrm{GL}(n,\mathbb{C})\ (n\geq1)$ and $\mathrm{SL}(n,\mathbb{C})\ (n\geq2)$ admit no algebraization by a complex scheme locally of finite type. When the genus is at least two, the same holds for their stable loci and hence for the corresponding Hitchin twistor spaces. Nevertheless, the Deligne--Hitchin twistor spaces admit $\mathbb{R}_{\mathrm{alg}}$-definable complex analytic structures when the rank or the genus is one.
\end{abstract}

\section{Introduction}\label{sec:introduction}

The study of Higgs bundles on a smooth complex projective curve started from Hitchin's seminal work on the famous self-duality equations, which also
gave a hyperK\"ahler structure on the relevant smooth moduli spaces in rank two \cite{Hit87}. Together with the fundamental work of Donaldson, Corlette, and Simpson \cite{Don87,Cor88,Sim92}, this yields the celebrated nonabelian Hodge correspondence between semisimple local systems and polystable Higgs bundles of degree zero. The algebraic counterpart was developed by Simpson, who constructed the corresponding algebraic moduli spaces \cite{Sim94a,Sim94b} and showed that the correspondence identifies their underlying topological spaces \cite[Theorem~7.18]{Sim94b}, while on the smooth locus their complex structures lie in the same hyperK\"ahler family \cite{Sim97}.

The moduli-theoretic description of this family begins with Deligne's introduction of 
$\lambda$-flat bundles, which interpolate between Higgs bundles and flat bundles. Simpson developed the theory further, constructing the Hodge moduli spaces of $\lambda$-flat bundles, and gluing the analytification of the Hodge moduli space to its counterpart for the complex conjugate curve \cite[Section~4]{Sim97} (see also \cite{Hua20a,Hua20b}). The resulting \emph{Deligne--Hitchin twistor space} has since been investigated through its Torelli properties \cite{BGHL09}, its automorphism groups in rank one \cite{BH17}, and the holomorphic symplectic geometry of suitable smooth loci in its spaces of holomorphic sections \cite{BBHR26}. Derived analytic versions of the Deligne--Hitchin construction have also been studied in \cite{FH24,KTY26}. Although the two Hodge moduli spaces are algebraic, the absence of a natural algebraic structure on the whole space was already noted in \cite{BGHL09,BH17}.

An analytic gluing of algebraic moduli spaces does not by itself settle whether the resulting complex space admits an algebraic structure, since a possible algebraization need not preserve the
given charts or the twistor projection. Therefore, to exclude every such algebraization, one needs an
obstruction which depends on the complex analytic structure itself. The goal of this paper is to prove the non-algebraicity of the full Deligne--Hitchin twistor spaces and their stable loci under the genus
assumptions below.

Let $X$ be a smooth complex projective curve of genus $g$, and let $G=\mathrm{GL}(n,\mathbb{C})$ with $n\geq1$, or $G=\mathrm{SL}(n,\mathbb{C})$ with
$n\geq2$. Denote by $\mathrm{TW}_{\rm DH}(X,G)$ the 
Deligne--Hitchin twistor space and by
$\mathrm{TW}_{\rm DH}(X,G)^s$ its stable locus,
where the determinant for $G=\mathrm{SL}(n,\mathbb{C})$
is fixed to be the trivial line $\lambda$-flat bundle
\cite[Section~2.3]{HHZ24}.
An \emph{algebraization} of a complex space is a biholomorphism
with the analytification of a complex scheme locally of finite type.
Our main result is the following.

\begin{thm}\label{thm:main}
If $g\geq1$, then $\mathrm{TW}_{\rm DH}(X,G)$ admits no
algebraization.
If $g\geq2$, then the same holds for
$\mathrm{TW}_{\rm DH}(X,G)^s$.
\end{thm}

In fact, we prove that none of these spaces admits a holomorphic embedding into the analytification of a quasi-separated complex algebraic space locally of finite type for which the analytification exists.
The full moduli spaces can be singular
\cite{Sim94b,Sim97}.

To prove the theorem, we first consider the rank one twistor space $Z_X:=\mathrm{TW}_{\rm DH}(X,\mathrm{GL}(1,\mathbb{C}))$ with twistor projection $\pi: Z_X\to\mathbb{P}^1$.
The lattice description of $Z_X$ in
\cite[proof of Theorem~4.3]{Sim97} and \cite[Section~6.7]{GX08} allows us to compute its entire meromorphic function field $\mathrm{Mer}(Z_X)$.

\begin{thm}\label{thm:rankone}
If $g\geq1$, then pullback by $\pi$ identifies
$\mathrm{Mer}(Z_X)$ with $\mathbb{C}(\mathbb{P}^1)$.
\end{thm}

This gives a different value from that asserted in \cite[Proposition~4.2]{BH17}, where $\mathrm{trdeg}_{\mathbb{C}}\mathrm{Mer}(Z_X)=2g+1$ for $g\geq2$ is claimed, and Remark \ref{rem:moishezon} explains why the extension theorem used there does not apply.

In Section \ref{sec:rankone}, we prove that if $F:T\to Y^{\mathrm{an}}$ is a holomorphic map with finite fibers from a nonempty connected complex manifold $T$ to the analytification of a quasi-separated complex algebraic space $Y$ locally of finite type, where $Y^{\mathrm{an}}$ is assumed to exist, then
$\dim T\leq\mathrm{trdeg}_{\mathbb{C}}\mathrm{Mer}(T)$. Applying this inequality to direct sums of line $\lambda$-flat bundles gives the required obstruction for the full moduli spaces.

Since these direct sums are strictly polystable in higher rank, in Section \ref{sec:stable}, following \cite[Section~7]{HT03}, we use cyclic \'etale covers to construct families of direct images. Removing the loci with nontrivial deck stabilizers yields stability while preserving the meromorphic function field, so that, after correcting the determinant, the same obstruction applies. The quasi-separated hypothesis is essential to the assertion about algebraic spaces: in rank one, the lattice quotient defines a nowhere quasi-separated algebraic space whose associated analytic \'etale quotient is $Z_X$ (Remark \ref{rem:exotic-boundary}).

The non-algebraicity in Theorem \ref{thm:main} does not preclude the existence of a definable complex analytic
structure. In Section \ref{sec:definable}, we use the lattice description in rank one to construct finite semialgebraic complex analytic atlases in the sense of Bakker--Brunebarbe--Tsimerman \cite{BBT23}. Together with the relative symmetric product description of the Hodge moduli spaces on an elliptic curve \cite{FT17}, this gives the following result.

\begin{thm}\label{thm:definable}
Suppose that $g\geq1$. If $G=\mathrm{GL}(1,\mathbb{C})$ or $g=1$, then $\mathrm{TW}_{\rm DH}(X,G)$ admits an
$\mathbb{R}_{\mathrm{alg}}$-definable complex analytic structure for which the twistor projection is definable.
\end{thm}

The structures in Theorem \ref{thm:definable} are not
required to preserve the standard algebraic definable structures on the Hodge charts. We discuss their relation to Hitchin twistor spaces and the scope of the construction in Remarks \ref{rem:def-hitchin} and \ref{rem:def-scope}.

\section{Meromorphic functions and algebraization}\label{sec:rankone}\label{sec:obstructions}

As objects interpolating between Higgs bundles and flat bundles, $\lambda$-flat bundles were introduced by Deligne and subsequently developed by Simpson. They have since proved to be powerful tools in various areas of mathematics. We begin this section by collecting some basic definitions; for further details, we refer the reader to \cite{Sim94a,Sim94b,Sim97,Sim22,Sim24,Hua20a,Hua20b,HH22,HHZ24}.
A \emph{$\lambda$-flat bundle} on $X$ is a triple
$(E,\bar\partial_E,D^\lambda)$, where $\mathcal{E}:=(E,\bar\partial_E)$ is a holomorphic vector bundle, with $\bar\partial_E$ denoting the holomorphic structure, and 
\[
D^\lambda:\mathcal{E}\to\mathcal{E}\otimes K_X
\]
is a $\mathbb{C}$-linear morphism satisfying the $\lambda$-twisted Leibniz rule
\[
D^\lambda(fs)=fD^\lambda(s)+\lambda s\otimes df.
\]
On a curve, the integrability condition for $D^\lambda$ is automatic. When $\deg E=0$, the  $\lambda$-flat bundle is called \emph{semistable}  (respectively, \emph{stable}) if every nonzero proper $D^\lambda$-invariant holomorphic subbundle has degree at most zero (respectively, less than zero). It is called \emph{polystable} if it is a direct sum of stable
$\lambda$-flat bundles of degree zero.
A rank one object is a \emph{line $\lambda$-flat bundle}, which we denote by $L$ when its operators are understood. The trivial object is denoted by $\mathbf{1}=(X\times\mathbb{C},\bar\partial,\lambda\partial)$.

The Hodge moduli space $\mathcal{M}_{\mathrm{Hod}}(X,G)$ parametrizes S-equivalence classes of semistable $\lambda$-flat bundles of degree zero as $\lambda$ varies in $\mathbb{C}$; for $G=\mathrm{SL}(n,\mathbb{C})$, the determinant is fixed to be $\mathbf{1}$ \cite{Sim94a,Sim97}. At $\lambda=0$, these are Higgs bundles, with trace-free Higgs field in the fixed determinant case.

For $\lambda\ne0$, after rescaling $D^\lambda$ by $\lambda^{-1}$ and applying the Riemann--Hilbert correspondence
\cite[Theorem~9.11]{Sim94b}, the analytic Hodge moduli space restricted to
$\mathbb{C}^*$ is identified with
$\mathcal{M}_{\mathrm{B}}(X,G)^{\mathrm{an}}
\times\mathbb{C}^*$, namely,
\[
\mathcal{M}_{\mathrm{Hod}}(X,G)^{\mathrm{an}}
\big|_{\mathbb{C}^*}
\cong
\mathcal{M}_{\mathrm{B}}(X,G)^{\mathrm{an}}
\times\mathbb{C}^*.
\]
Here $\mathcal{M}_{\mathrm{B}}(X,G)$ is the Betti moduli space of $G$-local systems on $X$, also known as the character variety. Deligne's gluing identifies $([\rho],\lambda)$ in the Hodge moduli space of $X$ with $([\rho],\lambda^{-1})$ in that of $\overline X$, keeping the local system unchanged \cite[Section~3.2]{Sim08}, and hence yields the \emph{Deligne--Hitchin twistor space} $\mathrm{TW}_{\mathrm{DH}}(X,G)$, together with a morphism
\[
\pi:\mathrm{TW}_{\mathrm{DH}}(X,G)
\longrightarrow\mathbb{P}^1
\]
(see \cite[Section~4]{Sim97} and \cite[Section~2.2]{HHZ24}).

For $g\geq2$, let $\mathrm{M}_{\mathrm{sol}}(X,G)^s$ denote the moduli space of irreducible solutions to Hitchin's equations. It is a smooth hyperK\"ahler manifold, with trivial determinant connection for $G=\mathrm{SL}(n,\mathbb{C})$. Its \emph{Hitchin twistor space} $\mathrm{TW}_{\mathrm{H}}(X,G)^s$ is biholomorphic to $\mathrm{TW}_{\mathrm{DH}}(X,G)^s$ over $\mathbb{P}^1$ \cite[Theorem~4.2]{Sim97} (see also \cite[Proposition~2]{Sim22} and \cite[Sections~4.1--4.2]{BBHR26}). The negative inversion convention of \cite[Section~1.2]{Sim22} is converted to ours by simultaneously negating the parameter and the operator on the conjugate Hodge chart. This comparison identifies the preferred sections and is compatible with the condition on the fixed determinant; by contrast, the full Deligne--Hitchin twistor space may contain singular polystable points. Consequently, Theorem \ref{thm:main} also rules out an algebraization of this Hitchin twistor space.

Assume throughout this section that $g\geq1$, and let $\Gamma=2\pi iH^1(X,\mathbb{Z})\subset H^1(X,\mathbb{C})$. In the harmonic lattice description \cite[proof of Theorem~4.3]{Sim97} (see also \cite[Sections 6.3, 6.7]{GX08}), a line $\lambda$-flat bundle is represented by a triple $(X\times\mathbb{C},\bar\partial+a,\lambda\partial+b)$, where $a\in H^{0,1}(X)$ and $b\in H^{1,0}(X)$, and the remaining gauge equivalence is
\begin{align}\label{eq:lattice-action}
 (\lambda,a,b)\longmapsto
 (\lambda,a+\gamma^{0,1},b+\lambda\gamma^{1,0}),
 \qquad \gamma\in\Gamma.
\end{align}
For $\mu=\lambda^{-1}$, the same local system determines the triple $(\overline X\times\mathbb{C},\partial+\mu b,\mu\bar\partial+\mu a)$ on $\overline X$. After exchanging the two summands on this chart, the transition is given by $(a_\infty,b_\infty)=\mu(a,b)$, which realizes the algebraic total space as $E_{2g}=\mathrm{Tot}\bigl(\mathcal{O}_{\mathbb{P}^1}(1)^{\oplus2g}\bigr)$. Consequently, the rank one twistor space is
\[
Z_X:=\mathrm{TW}_{\mathrm{DH}}\bigl(X,\mathrm{GL}(1,\mathbb{C})\bigr)=E_{2g}^{\mathrm{an}}/\Gamma.
\]

We shall also need this construction for a nonzero rational sub-Hodge structure $W_{\mathbb{Q}}\subset H^1(X,\mathbb{Q})$, with $W^{p,q}=W_{\mathbb{C}}\cap H^{p,q}(X)$. If
$\dim_{\mathbb{Q}}W_{\mathbb{Q}}=2d$, then restriction of the transition to $W^{0,1}\oplus W^{1,0}$ defines
$E_W\cong\mathrm{Tot}(\mathcal{O}_{\mathbb{P}^1}(1)^{\oplus2d})$,
on which $\Gamma_W=2\pi i(W_{\mathbb{Q}}\cap H^1(X,\mathbb{Z}))$ acts via \eqref{eq:lattice-action}. Denote by $Z_W$ the analytic quotient and by
$\pi:Z_W\to\mathbb{P}^1$ its projection.

\begin{prop}\label{prop:invariants}
The quotient $Z_W$ is a connected complex manifold whose universal cover is $E_W^{\mathrm{an}}$. Moreover, pullback by $\pi$ identifies $\mathrm{Mer}(Z_W)$ with $\mathbb{C}(\lambda)$.
\end{prop}

\begin{proof}
For $\gamma\in\Gamma_W$, one has
$\gamma^{1,0}=-\overline{\gamma^{0,1}}$, and the projection of $iW_{\mathbb{R}}$ to $W^{0,1}$ is an isomorphism of real vector spaces, so $\Gamma_W^{0,1}$ is a lattice. With compatible Hodge
norms, the metric $(\lVert a\rVert^2+\lVert b\rVert^2)/(1+|\lambda|^2)$ on $E_W$ gives the translation section associated to $\gamma$
the constant norm $\lVert\gamma^{0,1}\rVert$. Then the action is free and properly discontinuous, including over zero and infinity, which proves that $Z_W$ is a complex manifold. The total space
$E_W^{\mathrm{an}}$ retracts onto $\mathbb{P}^1$, so it is simply connected and serves as the universal cover of $Z_W$.

For $m=2d$, a point of
$E_m=\mathrm{Tot}(\mathcal{O}_{\mathbb{P}^1}(1)^{\oplus m})$ is a line $\ell\subset\mathbb{C}^2$ together with a linear map $\varphi:\ell\to\mathbb{C}^m$.
For $0\ne u=(u_0,u_1)\in\ell$, the assignment $(\ell,\varphi)\mapsto[u_0:u_1:\varphi(u)]$ gives the standard identification \cite{Ver14}
$$
 E_m\cong\mathbb{P}^{m+1}\setminus\{u_0=u_1=0\}
 =\mathbb{P}^{m+1}\setminus\mathbb{P}^{m-1}.
$$
By the meromorphic extension theorem \cite[Chapter~II, Theorem~8.11]{Dem12}, every meromorphic function on $E_m^{\mathrm{an}}$ therefore extends across the boundary of codimension two. The extended function is rational by Chow's theorem applied to its graph \cite[Chapter~II, Theorem~8.10]{Dem12},
which gives $\mathrm{Mer}(E_m^{\mathrm{an}})=\mathbb{C}(E_m)$.

Let $q:E_W^{\mathrm{an}}\to Z_W$ be the quotient map. Since $q$ is a covering map, every $\Gamma_W$-invariant meromorphic function on $E_W^{\mathrm{an}}$ descends uniquely to $Z_W$. Together with the preceding extension argument, this gives
\[
q^*\mathrm{Mer}(Z_W)
=
\mathrm{Mer}(E_W^{\mathrm{an}})^{\Gamma_W}
=
\mathbb{C}(E_W)^{\Gamma_W}.
\]
It therefore remains to compute the invariant field on the right.

Let $K=\mathbb{C}(\lambda)$ and choose an integral basis $\gamma_1,\ldots,\gamma_m$ of $\Gamma_W$. Since these vectors form a complex basis of $W_{\mathbb{C}}$ and
$\xi\mapsto(\xi^{0,1},\lambda\xi^{1,0})$ is an isomorphism over $K$, the corresponding translation vectors form a $K$-basis of the generic fiber. In suitable affine coordinates $t_1,\ldots,t_m$, translation by $\gamma_j$ sends $t_j$ to $t_j+1$ and fixes the other coordinates. A rational function in one variable over a field of
characteristic zero cannot have period one unless it is constant: any finite pole would have infinitely many translates, while a nonconstant polynomial is not periodic. Applying this argument to each variable gives $\mathbb{C}(E_W)^{\Gamma_W}=K$, and hence $q^*\mathrm{Mer}(Z_W)=K$, as required.

\end{proof}

Taking $W_{\mathbb{Q}}=H^1(X,\mathbb{Q})$ then proves
Theorem \ref{thm:rankone}.

\begin{rem}\label{rem:moishezon}
The transcendence degree in Theorem \ref{thm:rankone} is one, contrary to the value $2g+1$ asserted for $g\geq2$ in \cite[Proposition~4.2]{BH17}. The theorem invoked at the final step of that proof requires simple connectedness \cite[Theorem~3.4]{Ver14}, so it does not apply to $Z_X$, whose fundamental group is $\Gamma\cong\mathbb{Z}^{2g}$. Indeed, the norm estimate above gives a sufficiently small tube around the zero section that maps injectively to $Z_X$. Choose a linear functional $\ell$ with $\ell(\gamma^{0,1})\ne0$ for some $\gamma\in\Gamma$. Then $\ell(a)$ is meromorphic on the covering,
with expression $\ell(a_\infty)/\mu$ at infinity, and defines a meromorphic function on the image of this tube. If it extended globally to $Z_X$, its pullback would coincide with $\ell(a)$ by the identity theorem, but this contradicts the fact that it changes by $\ell(\gamma^{0,1})$ under \eqref{eq:lattice-action}.
\end{rem}

For quasi-separated algebraic spaces we use analytification through
\'etale presentations whenever it exists
\cite[Sections~1.1, 1.5, 2.2]{CT09}. For nonseparated schemes, affine
analytifications are glued with non-Hausdorff analytic spaces allowed.
We now compare the dimension of an algebraic image with the field
of meromorphic functions on its source.

\begin{lem}\label{lem:pullback}
Let $T$ be a nonempty connected complex manifold and let
$F:T\to Y^{\mathrm{an}}$ be a holomorphic map, where $Y$ is a quasi-separated complex algebraic space locally of finite type whose analytification exists. Then the reduced Zariski closure $V$ of $F(T)$ is
irreducible, and pullback induces an injection
$\mathbb{C}(V)\hookrightarrow\mathrm{Mer}(T)$.
If $F$ has finite fibers, then
\[
\dim T\leq\mathrm{trdeg}_{\mathbb{C}}\,\mathrm{Mer}(T).
\]
\end{lem}

\begin{proof}
By \cite[Tag~03IQ]{Sta}, the reduced closed subspace structure on $V$ exists, and since $T$ is reduced, $F$ factors through $V^{\mathrm{an}}$. By the identity principle, the inverse image of every proper closed algebraic subset of $V$ has empty interior, which implies that $V$ is irreducible and that every nonempty algebraic open subset of $V$ has analytically dense inverse image.
By \cite[Tag~06NH]{Sta}, $V$ contains a dense open subscheme. Choose a nonempty smooth affine open $V_0\subset V$ and denote by $K=\mathbb{C}(V_0)=\mathbb{C}(V)$ its function field, so we have
$\dim V_0=\mathrm{trdeg}_{\mathbb{C}}K$.

Let $f\in K$ be regular on a nonempty open $U\subset V_0$, and let $\mathcal{G}_f\subset V\times\mathbb{P}^1$ be the reduced algebraic closure of its graph. Let $\mathcal{H}\subset T\times\mathbb{P}^1$ be the reduced analytic inverse image of $\mathcal{G}_f^{\mathrm{an}}$ under $F\times\mathrm{id}_{\mathbb{P}^1}$, and put
$T_0=F^{-1}(U^{\mathrm{an}})$. The closure of the graph of $f\circ F$ over $T_0$ is analytic, since locally it is the union of those irreducible components of $\mathcal{H}$ which are not contained in
$(T\setminus T_0)\times\mathbb{P}^1$. This closure is closed in $T\times\mathbb{P}^1$, so its projection to $T$ is proper and is an isomorphism over the dense open $T_0$. It therefore defines a meromorphic function
$F^*f$ on the manifold $T$ \cite[Section~7]{Iva13}. On dense inverse images of common domains, these pullbacks preserve field operations and nonvanishing, which gives
an injection $\mathbb{C}(V)\hookrightarrow\mathrm{Mer}(T)$.

Finally, on the nonempty open $F^{-1}(V_0^{\mathrm{an}})$, the constant-rank theorem, applied at a point of maximal rank, together with the assumption that the fibers are finite, yields $\dim T\leq\dim V_0$. Combined with the field injection, this proves the desired inequality.
\end{proof}

\begin{proof}[Proof of Theorem \ref{thm:main} for the full spaces] Direct sum and duality give holomorphic maps over $\mathbb{P}^1$
$$
 \begin{aligned}
 Z_X&\longrightarrow\mathrm{TW}_{\rm DH}(X,\mathrm{GL}(n,\mathbb{C})),
 &L&\longmapsto L\oplus\mathbf{1}^{\oplus(n-1)},\\
 Z_X&\longrightarrow\mathrm{TW}_{\rm DH}(X,\mathrm{SL}(n,\mathbb{C})),
 &L&\longmapsto L\oplus L^*\oplus\mathbf{1}^{\oplus(n-2)}.
 \end{aligned}
$$
Indeed, the harmonic parameters give local holomorphic families, whose direct sums and duals determine holomorphic classifying maps
by \cite[Corollary~5.6 and Lemma~5.7]{Sim94a} and \cite[Proposition~4.1 and the following remark]{Sim97}, and these maps agree under the Deligne--Hitchin gluing because the corresponding operations on local systems agree. Since every line $\lambda$-flat bundle is stable, uniqueness of the stable factors \cite[Section~3, pp.~89--90]{Sim94a} makes the first map injective and bounds the fibers of the second by two, corresponding to $L$ and $L^*$, at every parameter. If either target embedded holomorphically into an analytification as in Lemma \ref{lem:pullback}, then applying that lemma to the composite gives $2g+1=\dim Z_X\leq1$, contrary to $g\geq1$. Finally, every complex scheme locally of finite type is locally Noetherian and hence quasi-separated, which proves the non-algebraization assertion as well.
\end{proof}

\begin{rem}\label{rem:exotic-boundary}
For $g\geq1$, quasi-separatedness cannot be omitted if algebraic spaces are understood in the broader sense of \cite[Tag~025Y]{Sta}. Every nonzero lattice vector acts on $E_{2g}$ by translation
along a nowhere-zero algebraic section, so the action is free and its quotient
$Q_X=E_{2g}/\Gamma$ is a smooth algebraic space locally of finite type over $\mathbb{C}$ \cite[Tag~02Z2]{Sta}.
For a nonempty open $Q'\subset Q_X$, its inverse image $E'$ is quasi-compact because $E_{2g}$ is Noetherian, whereas $E'\times_{Q'}E'=\coprod_{\gamma\in\Gamma}E'$ is not. Thus the diagonal of $Q'$ is not quasi-compact.

The associated analytic \'etale quotient sheaf is represented by $Z_X$, independently of the presentation. Indeed, for an affine \'etale chart $U\to Q_X$, the scheme $W=U\times_{Q_X}E_{2g}$ is \'etale over both factors and surjective over $U$. It is separated, since its map to the separated scheme $U\times E_{2g}$ is a monomorphism. The map $W^{\mathrm{an}}\to Z_X$ descends to a local isomorphism $U^{\mathrm{an}}\to Z_X$, because its two pullbacks to $(W\times_U W)^{\mathrm{an}}$ agree. For two charts, the comparison map
$(U\times_{Q_X}U')^{\mathrm{an}}
\to U^{\mathrm{an}}\times_{Z_X}(U')^{\mathrm{an}}$ is a local isomorphism and is bijective on points, since
$Q_X(\mathbb{C})=E_{2g}(\mathbb{C})/\Gamma$.
It is therefore an isomorphism, which proves the assertion. This realization through analytic \'etale quotients lies outside the quasi-separated convention for ordinary analytification in
\cite[Section~1.5]{CT09}.
\end{rem}

\section{The stable fixed determinant locus}\label{sec:stable}

Assume that $g\geq2$ and $n\geq2$, and choose a connected cyclic \'etale cover $p:Y\to X$ of degree $n$, with deck group $\Delta$, as in \cite[Section~7]{HT03}. Then $Y$ is a smooth complex projective curve of genus $h=1+n(g-1)$. The trace map respects the Hodge decomposition and satisfies $\mathrm{Tr}_p\circ p^*=n$, so
$W_{\mathbb{Q}}=\ker(\mathrm{Tr}_p:H^1(Y,\mathbb{Q})\to H^1(X,\mathbb{Q}))$ has dimension $2d$, where $d=(n-1)(g-1)$.
Let $\mathcal{P}=E_W^{\mathrm{an}}/\Gamma_W$ be the quotient in Proposition \ref{prop:invariants}, where
$\Gamma_W=2\pi i(W_{\mathbb{Q}}\cap H^1(Y,\mathbb{Z}))$ and
$E_W\cong\mathrm{Tot}(\mathcal{O}_{\mathbb{P}^1}(1)^{\oplus2d})$.
This is a connected complex manifold of dimension $2d+1$ with
$\mathrm{Mer}(\mathcal{P})=\mathbb{C}(\lambda)$.

The induced map $\mathcal{P}\to Z_Y$ is injective. Indeed, if two representatives with trace zero differ by an integral translation $\gamma$, then the translation section of $\mathrm{Tr}_p\gamma$ vanishes at their parameter, so the norm computation in the proof of Proposition \ref{prop:invariants} gives $\mathrm{Tr}_p\gamma=0$ and hence
$\gamma\in\Gamma_W$. The deck group $\Delta$ preserves $W_{\mathbb{Q}}$ and
$\Gamma_W$, and hence acts holomorphically on $\mathcal{P}$ over $\mathbb{P}^1$. Denote by $B=\bigcup_{1\ne\tau\in\Delta}\mathcal{P}^{\tau}$ the union of the fixed loci for the action on line $\lambda$-flat bundles,
and let $T=\mathcal{P}\setminus B$.

\begin{lem}\label{lem:fixed-loci}
The manifold $T$ is nonempty and connected, and satisfies
$\mathrm{Mer}(T)=\mathbb{C}(\lambda)$.
\end{lem}

\begin{proof}
Let $\tau\in\Delta$ have order $e>1$. Since
$H^1(Y,\mathbb{C})^\tau\cong H^1(Y/\langle\tau\rangle,\mathbb{C})$ and the trace is surjective on this invariant part, we have
$$
\dim W_{\mathbb{C}}^\tau=2(n/e-1)(g-1).
$$
For $z\in\mathcal{P}^\tau$ and $t=\pi(z)$, the tangent sequence
$$
0\longrightarrow(E_W)_t\longrightarrow T_z\mathcal{P}\longrightarrow T_t\mathbb{P}^1\longrightarrow0
$$
is an exact sequence of $\langle\tau\rangle$-representations. Its vertical term is isomorphic to $W_{\mathbb{C}}$ as a representation, since lattice translations have identity differential in the fiber directions, and its quotient is trivial. Taking invariants and using finite order holomorphic linearization therefore gives
$$
\mathrm{codim}_{\mathcal{P}}\mathcal{P}^\tau
=\dim W_{\mathbb{C}}-\dim W_{\mathbb{C}}^\tau =2(n-n/e)(g-1)\geq2.
$$
This holds at every parameter, including zero and infinity. Thus $B$ is a closed analytic subset of codimension at least two, so its complement $T$ is nonempty and connected. Meromorphic extension across $B$ gives $\mathrm{Mer}(T)=\mathrm{Mer}(\mathcal{P})=\mathbb{C}(\lambda)$.
\end{proof}

For a fixed parameter, semistable degree zero $\lambda$-flat bundles form an abelian category whose simple objects are the stable ones \cite[Section~3]{Sim94a}, so a subobject of a direct sum of pairwise non-isomorphic line objects is the sum of a subset of those factors.

\begin{lem}\label{prop:direct-image-stable}
Let $L$ and $L'$ be line $\lambda$-flat bundles of degree zero on $Y$. Then $p_*L$ is semistable, and it is stable when the $\Delta$-orbit of $L$ is free. Moreover, $p_*L\cong p_*L'$ implies $L'\cong\sigma^*L$ for some $\sigma\in\Delta$.
\end{lem}

\begin{proof}
The canonical isomorphism $K_Y\cong p^*K_X$ induces a $\lambda$-flat bundle structure on $p_*L$, and the sheet decomposition gives
\begin{align}\label{eq:galois-splitting}
p^*p_*L\cong\bigoplus_{\sigma\in\Delta}\sigma^*L.
\end{align}
Riemann--Roch gives $\deg p_*L=0$, and the semistability of the right-hand side implies that every invariant subbundle $A\subset p_*L$ satisfies $n\deg A=\deg p^*A\leq0$. If the orbit is free and $\deg A=0$, then $p^*A$ is a subobject in the preceding abelian category, hence the sum of a subset of the distinct simple factors in
\eqref{eq:galois-splitting}. Descent makes that subset invariant under the transitive deck action, so $A$ is zero or all of $p_*L$, which proves stability. Finally, an isomorphism of direct images identifies the multisets of simple factors after pullback, and hence $L'$ is a deck transform of $L$.
\end{proof}

Let $\delta=\mathrm{det}(p_*\mathbf{1})$ be the sign character of the cover. Since the unitary character group is $\mathrm{U}(1)^{2g}$, there is a unitary line local system $M$ with a fixed isomorphism $M^{\otimes n}\cong\delta^{-1}$. We also denote by $M$ its preferred twistor section, whose operator in each Hodge chart is the chart parameter times the corresponding part of its flat unitary connection \cite[Section~4]{Sim97}.

\begin{prop}\label{prop:prym-map}
The assignment $L\mapsto M\otimes p_*L$ defines a holomorphic map $\Phi:\mathcal{P}\to\mathrm{TW}_{\rm DH}(X,\mathrm{SL}(n,\mathbb{C}))$ over $\mathbb{P}^1$, whose restriction to $T$ takes values in the stable locus and has fibers of cardinality at most $n$.
\end{prop}

\begin{proof}
On each Hodge chart of $E_W$, the triples
$(Y\times\mathbb{C},\bar\partial_Y+a,\lambda\partial_Y+b)$ form a holomorphic relative family whose norm is trivial, since $\mathrm{Tr}_p a=\mathrm{Tr}_p b=0$.
Finite \'etale direct image commutes with analytic base change, and the sheet decomposition gives, compatibly with the $\lambda$-flat structures,
\begin{align}\label{eq:determinant-correction}
\mathrm{det}(M\otimes p_*L)
\cong M^{\otimes n}\otimes\mathrm{Nm}_p(L)\otimes\delta
\cong\mathbf{1}.
\end{align}
The resulting families are semistable by
Lemma~\ref{prop:direct-image-stable}, since tensoring by $M$ preserves semistability.

Choose harmonic representatives of an integral basis $\gamma_j$ of $\Gamma_W$ and unitary gauges $g_j$ with $g_j^{-1}dg_j=\gamma_j$. The traces of these harmonic forms vanish, so
$d\log\mathrm{Nm}_p(g_j)=\mathrm{Tr}_p\gamma_j=0$ and the functions $\mathrm{Nm}_p(g_j)$ are constant. After rescaling the $g_j$ by unitary constants, we may assume that $\mathrm{Nm}_p(g_j)=1$. Extending these gauges multiplicatively to $\Gamma_W$ gives compatible lattice
identifications on both Hodge charts which preserve the determinant trivializations in
\eqref{eq:determinant-correction}. By the analytic moduli property used above, the local families therefore define holomorphic maps on the two Hodge charts of $\mathcal{P}$. These maps glue because over nonzero parameters both constructions give the same induced local system on $X$.

By Lemma~\ref{prop:direct-image-stable}, the restriction to $T$ takes values in the stable locus, and equality of two images implies $L'\cong\sigma^*L$ for some
$\sigma\in\Delta$. Together with the injectivity of $\mathcal{P}\to Z_Y$, this gives the bound $n$ on every fiber.
\end{proof}

\begin{proof}[Proof of Theorem \ref{thm:main} for the stable loci] 
Since $\dim T=2(n-1)(g-1)+1>1$ and
$\mathrm{Mer}(T)=\mathbb{C}(\lambda)$, composing $\Phi|_T$ with a hypothetical embedding contradicts Lemma \ref{lem:pullback}, which proves the assertion for the stable $\mathrm{SL}(n,\mathbb{C})$ Deligne--Hitchin twistor space. The same argument using deck orbits gives finite fibers after forgetting the determinant, which proves the assertion for the stable $\mathrm{GL}(n,\mathbb{C})$ Deligne--Hitchin twistor space. Since every rank one object is stable, the remaining case follows from Proposition \ref{prop:invariants} and Lemma \ref{lem:pullback}.
\end{proof}

\begin{rem}\label{rem:exceptions}
For genus zero and $G=\mathrm{GL}(n,\mathbb{C})$ or $\mathrm{SL}(n,\mathbb{C})$, the equality $\pi_1(\mathbb{P}^1)=1$ and the Riemann--Hilbert isomorphism \cite[Theorem~9.11]{Sim94b} make the de Rham moduli space a reduced point. Since the Hodge moduli space is separated, the trivial family gives a closed section, which is an isomorphism
over $\mathbb{C}^{*}$. Its defining ideal is locally killed by powers of $\lambda$, whereas flatness of the parameter map \cite[Corollary~9.2]{Sim97} makes multiplication by $\lambda$ injective, so the ideal vanishes. Thus the Hodge moduli space is $\mathbb{A}^1$ and the Deligne--Hitchin twistor space
is $\mathbb{P}^1$. For genus one, commuting monodromy matrices have a common eigenvector, so the stable loci of rank greater than one are empty, including the Higgs fibers by \cite[Corollary~1.3]{Sim92}. Finally, $\mathrm{TW}_{\rm DH}(X,\mathrm{SL}(1,\mathbb{C}))=\mathbb{P}^1$.
\end{rem}

\begin{rem}\label{rem:open-curves}
For an open curve, Simpson gave a rank two construction using logarithmic $\lambda$-flat bundles and a Hecke gauge groupoid \cite{Sim22}, which extends to higher rank as an analytic groupoid under the hypotheses in \cite[Theorem~1.1]{Sim24}. Algebraic examples occur in this setting. For $\mathbb{P}^1\setminus\{0,\infty\}$, the rank one logarithmic model with trivial extension has coordinates $(\lambda,\alpha)$ and transition $(\lambda,\alpha)\mapsto(\lambda^{-1},-\lambda^{-2}\alpha)$, so it is $\mathrm{Tot}(\mathcal{O}_{\mathbb{P}^1}(2))$
\cite[Sections~4.1--4.2]{Sim08}. This algebraic model retains the logarithmic extension and must be distinguished from its further quotient by meromorphic gauge transformations \cite[Sections 4.4, 5.1]{Sim08}, \cite[Section~1.3]{Sim22}.

There is also an algebraic example after taking the Hecke quotient. For rank one on
$\mathbb{A}^1=\mathbb{P}^1\setminus\{\infty\}$, fix a framing at one point and allow logarithmic extensions of every degree, as an application of the construction in \cite[Sections~2.1, 3.2, 3.3]{Sim24}. On each $\mathcal{O}_{\mathbb{P}^1}(k)$, extending
$\lambda d$ gives the unique logarithmic $\lambda$-flat bundle structure, since
$H^0(\mathbb{P}^1,K_{\mathbb{P}^1}(\infty))=0$. The hypotheses in \cite[Hypothesis~2.1]{Sim24} are automatic in this case, since endomorphisms are scalars and the Higgs spectral curve is the zero section.

For an algebraic family over a parameter space $S$, the degree is locally constant, and on a locus where it equals $k$, the underlying line bundle has the form
\[
\mathcal{L}\cong
\mathrm{pr}_1^*\mathcal{O}_{\mathbb{P}^1}(k)
\otimes\mathrm{pr}_2^*N
\]
for a line bundle $N$ on $S$. The framing trivializes $N$ \cite[Tag~0B9R]{Sta}, and the same vanishing after base change gives uniqueness of the relative logarithmic
$\lambda$-flat bundle structure. These families therefore form copies of $\mathbb{A}^1_\lambda$ indexed by
$k\in\mathbb{Z}$, and Hecke modifications shift $k$ freely and transitively, so their framed groupoid quotient is $\mathbb{A}^1_\lambda$. The two quotient charts glue to $\mathbb{P}^1$ by parameter inversion. These open curve constructions are outside the scope of Theorem \ref{thm:main}.
\end{rem}

\section{Definable structures}\label{sec:definable}

In this section, we consider \emph{definable complex analytic spaces} in the sense of \cite[Section~2]{BBT23} for the o-minimal structure $\mathbb{R}_{\mathrm{alg}}$, so that definable means semialgebraic, with real parameters allowed. A \emph{finite semialgebraic complex analytic atlas} has finitely many charts modeled on complex analytic subspaces of semialgebraic open subsets of complex affine spaces, cut out by finitely many semialgebraic holomorphic functions. The overlaps and holomorphic transition maps are also required to be
semialgebraic.

\begin{prop}\label{prop:def-rankone}
The quotient $Z_W$ in Proposition \ref{prop:invariants} admits a finite semialgebraic complex analytic atlas for which the projection to $\mathbb{P}^1$ and the fiberwise group operations are definable. The maps induced by morphisms of integral Hodge structures are definable in these atlases.
\end{prop}

\begin{proof}
Let $V=W^{0,1}$ and let $\Lambda\subset V$ be the projection of $\Gamma_W$. In the finite parameter chart, the lattice acts by $(\lambda,a,b)\mapsto(\lambda,a+c,b-\lambda\overline c)$ for $c\in\Lambda$. Consider the real coordinates $(u,v)\in V\oplus\overline V$ given by $a=u+\lambda\overline v$ and $b=v-\lambda\overline u$, which satisfy
$$
u=\frac{a-\lambda\overline b}{1+|\lambda|^2},\qquad
v=\frac{b+\lambda\overline a}{1+|\lambda|^2}.
$$
Then the lattice acts by $(u,v)\mapsto(u+c,v)$. Choose finitely many bounded semialgebraic open balls in $V$ whose images cover $V/\Lambda$ and on each of which the quotient map is injective. Their inverse images under $u$ are semialgebraic open sets in the holomorphic coordinates $(\lambda,a,b)$, and the quotient map restricts to a biholomorphism on each of them. Their images cover the quotient over the finite parameter line.

On the other parameter chart, the holomorphic coordinates are $(\mu,a_\infty,b_\infty)=(\lambda^{-1},\lambda^{-1}a,\lambda^{-1}b)$, and
$$
u=\frac{\overline\mu a_\infty-\overline b_\infty}{1+|\mu|^2},\qquad
v=\frac{\overline\mu b_\infty+\overline a_\infty}{1+|\mu|^2}.
$$
The same balls therefore give charts at infinity. On an overlap, the lattice corrections belong to the intersection of $\Lambda$ with the difference of two chosen balls, which is finite since the balls are bounded. The transitions are consequently given on finitely many semialgebraic domains by holomorphic lattice translations, possibly composed with $(a_\infty,b_\infty)=\lambda^{-1}(a,b)$. Thus the transitions and the projection are semialgebraic. The real coordinates $u,v$ are used only to specify the chart domains.

The zero section and the fiberwise group operations are induced by zero, addition and negation in the harmonic parameters. Morphisms of integral Hodge structures are
linear in these parameters and preserve the lattices. The images and finite sums of the chosen bounded balls are bounded, so the same argument gives only finitely many lattice corrections in the target charts, which proves the remaining assertions.
\end{proof}

For a space $Y$ over $S$, denote by $Y_S^{\times n}=Y\times_S\cdots\times_S Y$ its $n$-fold fiber product over $S$.

\begin{thm}\label{thm:def-elliptic}
Suppose that $X$ has genus one. Let
$s_n: (Z_X)_{\mathbb{P}^1}^{\times n}\to Z_X$ denote fiberwise tensor product, and let $\mathbf{1}: \mathbb{P}^1\to Z_X$ be the unit section. Then there are complex analytic isomorphisms over $\mathbb{P}^1$
\begin{align*}
\mathrm{TW}_{\rm DH}(X,\mathrm{GL}(n,\mathbb{C}))
&\cong (Z_X)_{\mathbb{P}^1}^{\times n}/\mathfrak{S}_n,\\
\mathrm{TW}_{\rm DH}(X,\mathrm{SL}(n,\mathbb{C}))
&\cong s_n^{-1}(\mathbf{1})/\mathfrak{S}_n,
\end{align*}
where $s_n^{-1}(\mathbf{1})$ denotes the inverse image of the unit section. Both spaces admit finite semialgebraic complex analytic atlases for which the twistor projections are definable.
\end{thm}

\begin{proof}
Let $H$ be the rank one Hodge moduli space over $\mathbb{A}^1$, and choose an origin on $X$. For a scheme $S$ of finite type over $\mathbb{A}^1$, the relative Fourier--Mukai transform identifies families of semistable $\lambda$-flat bundles of rank $n$ and degree zero parametrized by $S$ with $S$-flat coherent sheaves on $H\times_{\mathbb{A}^1}S$ whose support is finite over $S$ and whose fibers have length $n$, compatibly with base change \cite[Theorem~3.14 and Section~5]{FT17}. For an affine open $\mathrm{Spec}(A)\subset H$, let $R=\mathbb{C}[\lambda]$, and let $\mathrm{Rep}_{n,R}(A)$ be the relative scheme of $n$-dimensional representations of $A$. After locally choosing a frame
of their finite pushforward to $S$, the transformed families supported in $\mathrm{Spec}(A)$ are parametrized by $\mathrm{Rep}_{n,R}(A)$, and changing the frame gives the action of $\mathrm{GL}(n,\mathbb{C})$ fixing $\lambda$. Since $H\to\mathbb{A}^1$ is smooth, $A$ is flat over $R$, so the corresponding affine quotient is determined by
$$
R[\mathrm{Rep}_{n,R}(A)]^{\mathrm{GL}(n,\mathbb{C})}\cong(A^{\otimes_R n})^{\mathfrak{S}_n}
$$
by \cite[Theorems~2 and~4]{Vac06}. The invariant-ring isomorphism is induced by restriction to diagonal representations, so the corresponding map of coarse moduli spaces is induced by direct sum of line $\lambda$-flat bundles. Since the support of a transformed family is finite over the parameter space, the condition that it be contained in a given affine open of $H$ is open and is preserved by semisimplification.
These support opens cover the moduli problem, because every finite subset of $H$ lies in an affine open. The affine quotient identifications are compatible with restriction, so they glue to an isomorphism
$$
\mathcal{M}_{\mathrm{Hod}}(X,\mathrm{GL}(n,\mathbb{C}))\cong H_{\mathbb{A}^1}^{\times n}/\mathfrak{S}_n
$$
over $\mathbb{A}^1$.

Under this isomorphism, determinant is induced by tensor product of the line factors. The equations of the unit section are invariant, so taking its inverse image
commutes with both the good quotient of the framed parameter space and the finite $\mathfrak{S}_n$-quotient, by exactness of invariants in characteristic zero
\cite[Section~1]{Sim94a}. A scalar automorphism acts on a determinant trivialization by its $n$th power, so choosing such a trivialization gives no additional coarse parameter \'etale locally on the base. This identifies the special linear Hodge moduli space with the quotient of the tensor product fiber.

The same argument applies to $\overline{X}$. Since the Riemann--Hilbert correspondence preserves direct sum and determinant, the identifications on the two Hodge charts commute with Deligne's gluing and yield the asserted analytic isomorphisms over $\mathbb{P}^1$. Proposition \ref{prop:def-rankone} makes the relative products, the tensor product fiber and the permutation actions definable. Hence \cite[Proposition~2.63]{BBT23} gives definable structures on both quotients. Transporting these structures through the analytic isomorphisms above gives the required atlases, and the projections remain definable because
the isomorphisms are over $\mathbb{P}^1$.
\end{proof}

Proposition \ref{prop:def-rankone} and Theorem \ref{thm:def-elliptic} prove Theorem \ref{thm:definable}.

\begin{rem}\label{rem:def-hitchin}
In rank one, the comparison isomorphism recalled in Section \ref{sec:rankone} also holds when $g=1$, so Proposition \ref{prop:def-rankone} gives a definable structure on the Hitchin twistor space \cite[Theorem~4.2]{Sim97}. When $g=1$ and $n>1$, the space $\mathrm{M}_{\rm sol}(X,G)^s$ of irreducible solutions is empty by Remark \ref{rem:exceptions}, whereas the smooth coarse moduli locus consists of direct sums of pairwise non-isomorphic line objects. Points with repeated factors are singular, since a single double collision has transverse germ $\mathbb{C}^2/\{\pm1\}$ and every other collision is a limit of such points, also with fixed determinant.

On this smooth locus, the trace pairing splits into the rank one pairings, and the first cohomology of a nontrivial rank one local system on $X$ vanishes. Thus direct sum identifies its $L^2$ metric with the
quotient of the product metric, restricted to the tensor product fiber in the special linear case. The Hitchin twistor space of this smooth coarse locus is therefore
the corresponding quotient in Theorem
\ref{thm:def-elliptic} with the images of the pairwise diagonals removed, and is definable.
\end{rem}

\begin{rem}\label{rem:def-scope}
For $g\geq2$ and $n\geq2$, direct sums of line $\lambda$-flat bundles do not cover the full Deligne--Hitchin twistor space, since Proposition \ref{prop:prym-map} gives stable objects of rank $n$. Therefore, the symmetric product argument does not settle definability of the full space or its stable locus in this range. Moreover, a definable family need not have a definable
classifying map to a moduli space with its standard algebraic definable structure, as the semialgebraic family of line bundles in
\cite[Sections~2.2--2.3]{EK26} shows.

For $g\geq1$, Deligne's gluing is not definable with respect to the standard algebraic definable structures on the punctured Hodge charts. Otherwise, both it and its inverse would be algebraic by \cite[Corollary~3.11]{BBT23}, and the two Hodge schemes would then glue to an algebraization, contrary to Theorem \ref{thm:main}. This obstruction already occurs in rank one, which shows why the existence of other definable atlases must be distinguished from compatibility with the standard Hodge charts.
\end{rem}

\bigskip

\noindent\textbf{Acknowledgments}. 
The precise meaning and scope of the non-algebraicity of Deligne--Hitchin twistor spaces have, in the author's view, not always been clear in the mathematical community, which has motivated him to provide a precise statement together
with a rigorous proof. He would like to thank, among others, Zhi Hu, Carlos Simpson, Hao Sun and Bin Xu for their helpful comments and discussions on various occasions. He is especially grateful to Zhi Hu for valuable suggestions concerning the definable structures studied in Section \ref{sec:definable}.

\bigskip

\noindent\textbf{Declarations on the use of generative AI}. During the development of this paper, the author used OpenAI's ChatGPT (GPT-5.6 and 6) as an interactive tool to explore a proof strategy for  stable loci as presented in Section \ref{sec:stable} and formulate some purely algebraic arguments. All mathematical arguments and references were drafted by the author, with revision, proof simplification, and language polishing by AI, and the author independently checked all mathematical arguments and takes full responsibility for the contents.

\Needspace{3\baselineskip}

\bigskip
\noindent\small{\textsc{School of Mathematics, Nanjing University}\\
Nanjing 210093, China}\\
\emph{E-mail address}: \texttt{pfhwangmath@gmail.com}


\begin{thebibliography}{99}

\bibitem{BBT23}
B.~Bakker, Y.~Brunebarbe and J.~Tsimerman,
o-minimal GAGA and a conjecture of Griffiths.
\emph{Invent. Math.} \textbf{232} (2023), 163--228.

\bibitem{BBHR26}
F.~Beck, I.~Biswas, S.~Heller and M.~R\"oser,
Geometry of the space of sections of twistor spaces with circle action.
\emph{SIGMA Symmetry Integrability Geom. Methods Appl.}
\textbf{22} (2026), Paper No.~008, 43 pp.

\bibitem{BGHL09}
I.~Biswas, T.~L.~G\'omez, N.~Hoffmann and M.~Logares,
Torelli theorem for the Deligne--Hitchin moduli space.
\emph{Comm. Math. Phys.} \textbf{290} (2009), 357--369.

\bibitem{BH17}
I.~Biswas and S.~Heller,
On the automorphisms of a rank one Deligne--Hitchin moduli space.
\emph{SIGMA.} \textbf{13} (2017), Paper No.~072, 19 pp.

\bibitem{CT09}
B.~Conrad and M.~Temkin,
Non-Archimedean analytification of algebraic spaces.
\emph{J. Alg. Geom.} \textbf{18} (2009), no.~4, 731--788.

\bibitem{Cor88}
K.~Corlette,
Flat $G$-bundles with canonical metrics.
\emph{J. Diff. Geom.} \textbf{28} (1988), no.~3, 361--382.

\bibitem{Dem12}
J.-P.~Demailly,
\emph{Complex analytic and differential geometry}.
Online monograph, version of 21 June 2012.
\url{https://www-fourier.univ-grenoble-alpes.fr/~demailly/manuscripts/agbook.pdf}.

\bibitem{Don87}
S.~K.~Donaldson,
Twisted harmonic maps and the self-duality equations.
\emph{Proc. London Math. Soc.} (3) \textbf{55} (1987), no.~1, 127--131.

\bibitem{EK26}
H.~Esnault and M.~Kerz,
There is no Definable Grauert Direct Image Theorem.
\href{https://arxiv.org/abs/2601.18711v2}{arXiv:2601.18711}.

\bibitem{FH24}
E. Franco and R. Hanson,
The Dirac--Higgs complex and categorification of $(BBB)$-branes.
\emph{Int. Math. Res. Not.} (2024), no. 19, 12919--12953.

\bibitem{FT17}
E.~Franco and P.~Tortella,
Moduli spaces of $\Lambda$-modules on abelian varieties.
\emph{Adv. Math.} \textbf{318} (2017), 459--496.

\bibitem{GX08}
W.~M.~Goldman and E.~Z.~Xia,
Rank one Higgs bundles and representations of fundamental groups of
Riemann surfaces.
\emph{Mem. Amer. Math. Soc.} \textbf{193} (2008), no.~904, viii+69 pp.

\bibitem{HT03}
T.~Hausel and M.~Thaddeus,
Mirror symmetry, Langlands duality, and the Hitchin system.
\emph{Invent. Math.} \textbf{153} (2003), no.~1, 197--229.

\bibitem{Hit87}
N.~J.~Hitchin,
The self-duality equations on a Riemann surface.
\emph{Proc. London Math. Soc.} (3) \textbf{55} (1987), no.~1, 59--126.

\bibitem{HH22}
Z. Hu and P. Huang, 
Simpson--Mochizuki correspondence for $\lambda$-flat bundles.
\emph{J. Math. Pures Appl.} \textbf{164} (2022),
57--92.

\bibitem{HHZ24}
Z.~Hu, P.~Huang and R.~Zong,
Generalized Deligne--Hitchin twistor spaces: construction and properties.
\emph{Bull. Sci. Math.} \textbf{191} (2024), Paper No.~103396.

\bibitem{Hua20a}
P.~Huang,
Non-Abelian Hodge theory and related topics.
\emph{SIGMA.} \textbf{16} (2020), Paper No.~029, 34 pp.

\bibitem{Hua20b}
P.~Huang,
\emph{Non-abelian Hodge theory and some specializations}.
Ph.D. thesis, Universit\'e C\^ote d'Azur and
University of Science and Technology of China, 2020.

\bibitem{Iva13}
S.~Ivashkovich,
Extension properties of complex analytic objects.
Max-Planck-Institut f\"ur Mathematik Preprint Series \textbf{15} (2013).

\bibitem{KTY26}
J. Kryczka, Y. Tanaka and S.-T. Yau,
Lagrangian correspondences of nonabelian Hodge type and shifted twistor structures.
\href{https://arxiv.org/abs/2606.30637v2}{arXiv:2606.30637}.

\bibitem{Sim92}
C.~T.~Simpson,
Higgs bundles and local systems.
\emph{Publ. Math. Inst. Hautes \'Etudes Sci.} \textbf{75} (1992), 5--95.

\bibitem{Sim94a}
C.~T.~Simpson,
Moduli of representations of the fundamental group of a smooth projective variety I.
\emph{Publ. Math. Inst. Hautes \'Etudes Sci.} \textbf{79} (1994), 47--129.

\bibitem{Sim94b}
C.~T.~Simpson,
Moduli of representations of the fundamental group of a smooth projective variety II.
\emph{Publ. Math. Inst. Hautes \'Etudes Sci.} \textbf{80} (1994), 5--79.

\bibitem{Sim97}
C. T. Simpson,
The Hodge filtration on nonabelian cohomology.
In \emph{Algebraic geometry---Santa Cruz 1995}, Proc. Sympos. Pure Math., vol.~62, Part~2, Amer. Math. Soc., Providence, RI, 1997, 217--281.

\bibitem{Sim08}
C.~T.~Simpson,
A weight two phenomenon for the moduli of rank one local systems on open varieties.
In \emph{From Hodge theory to integrability and TQFT: $tt^{*}$-geometry}, 
Proc. Sympos. Pure Math., vol.~78,
Amer. Math. Soc., Providence, RI, 2008, 175--214.

\bibitem{Sim22}
C.~T.~Simpson,
The twistor geometry of parabolic structures in rank two.
\emph{Proc. Indian Acad. Sci. Math. Sci.}
\textbf{132} (2022), Paper No.~54, 26 pp.

\bibitem{Sim24}
C.~T.~Simpson,
Twistor space for local systems on an open curve.
\emph{Int. J. Math.} \textbf{35} (2024), no.~9,
Paper No.~2441013, 22 pp.

\bibitem{Sta}
The Stacks Project Authors,
\emph{The Stacks Project}, 2026.

\bibitem{Vac06}
F.~Vaccarino,
Symmetric product as moduli space of linear representations.
\href{https://arxiv.org/abs/math/0608655}{arXiv:math/0608655}.

\bibitem{Ver14}
M.~Verbitsky,
Holography principle for twistor spaces.
\emph{Pure Appl. Math. Q.} \textbf{10} (2014), no.~2, 325--354.

\end{thebibliography}
\end{document}